\documentclass[preprint,12pt]{elsarticle}
\usepackage[margin=2.5cm]{geometry}
\usepackage{graphicx}
\usepackage{amssymb}
\usepackage[centertags]{amsmath}
\usepackage{amsfonts}
\usepackage{amssymb}
\usepackage{amsthm}
\usepackage[symbol*]{footmisc}
\usepackage{tikz}
\usepackage{tkz-euclide}
\usetkzobj{all}
\DefineFNsymbolsTM{myfnsymbols}{
  \textasteriskcentered *
  \textdagger    \dagger
  \textdaggerdbl \ddagger
}
\setfnsymbol{myfnsymbols}

\newtheorem{thm}{Theorem}[section]
\newtheorem{prop}[thm]{Proposition}
\newtheorem{lem}[thm]{Lemma}
\newtheorem{oss}[thm]{Remark}

\def\RR{{\mathbb{R}}}

\def\NN{{\mathbb{N}}}

\definecolor{green(ryb)}{rgb}{0.173, 0.627, 0.173}

\journal{Journal Name}

\begin{document}

\begin{frontmatter}


\title{Exact controllability to the ground state solution for evolution equations of parabolic type via bilinear control\footnote{This paper was partly supported by the INdAM National Group for Mathematical Analysis, Probability and their Applications.}}



\author{Fatiha Alabau Boussouira}
\address{Laboratoire Jacques-Louis Lions Sorbonne Universit\'{e}, Universit\'{e} de Lorraine, 75005, Paris, France

alabau@ljll.math.upmc.fr}
\author{Piermarco Cannarsa\footnote{This author acknowledges support from the MIUR Excellence Department Project awarded to the Department of Mathematics, University of Rome Tor Vergata, CUP E83C18000100006.}} 
\address{Dipartimento di Matematica, Universit\`{a} di Roma Tor Vergata, 00133, Roma, Italy 

cannarsa@mat.uniroma2.it}
\author{Cristina Urbani\footnote{This author is grateful to University Italo Francese (Vinci Project 2018).}}
\address{Gran Sasso Science Institute, 67100, L'Aquila, Italy

Laboratoire Jacques-Louis Lions, Sorbonne Universit\'{e}, 75005, Paris, France

 cristina.urbani@gssi.it}

\begin{abstract}
In a separable Hilbert space $X$, we study the linear evolution equation
\begin{equation*}
u'(t)+Au(t)+p(t)Bu(t)=0,
\end{equation*}
where $A$ is an accretive self-adjoint linear operator, $B$ is a bounded linear operator on $X$, and $p\in L^2_{loc}(0,+\infty)$ is a bilinear control.

We give sufficient conditions in order for the above control system to be locally controllable to the ground state solution, that is, the solution of the free equation ($p\equiv0$) starting from the ground state of $A$. We also derive global controllability results in large time and discuss applications to parabolic equations in low space dimension.
\end{abstract}

\end{frontmatter}

%\linenumbers

\section{Introduction}
In a separable Hilbert space $X$, consider the control system
\begin{equation}\label{u}
\left\{
\begin{array}{ll}
u'(t)+Au(t)+p(t)Bu(t)=0,& t>0\\\\
u(0)=u_0.
\end{array}\right.
\end{equation}
where $A:D(A)\subset X\to X$ is a linear self-adjoint maximal accretive operator on $X$, $B$ belongs to $\mathcal{L}(X)$, the space of all bounded linear operators on $X$, and $p(t)$ is a scalar function representing a bilinear control.

In the recent paper \cite{acu}, we have studied the stabilizability of \eqref{u} along the ground state solution of the free equation ($p\equiv0$). More precisely, denoting by $\{\lambda_k\}_{k\in\NN^*}$ the eigenvalues of $A$ and by $\{\varphi_k\}_{k\in\NN^*}$ the corresponding eigenfunctions, we call $\varphi_1$ the ground state of $A$ and $\psi_1(t)=e^{-\lambda_1 t}\varphi_1$ the ground state solution of the equation
\begin{equation*}
u'(t)+Au(t)=0.
\end{equation*}
In \cite[Theorem 3.4]{acu}, we have given sufficient conditions on $A$ and $B$ to ensure the superexponential stabilizability of \eqref{u} along $\psi_1$: for all $u_0$ in some neighborhood of $\varphi_1$ there exists a control $p\in L^2_{loc}([0,+\infty))$ such that the corresponding solution $u$ of \eqref{u} satisfies 
\begin{equation}\label{superex}
||u(t)-\psi_1(t)||\leq Me^{-(e^{\omega t}+\lambda_1 t)},\qquad \forall\,t\geq0
\end{equation}
for some constants $\omega,M>0$.
In the same paper, we have discussed several applications of the above result to parabolic operators. For instance, we have studied the stabilizability of   the equation
\begin{equation*}
u_t-u_{xx}+p(t)\mu(x)u=0
\end{equation*}
with Dirichlet or Neumann boundary conditions, as well as the equation with variable coefficients
\begin{equation*}
u_t-((1+x)^2u_x)_x+p(t)Bu=0,
\end{equation*}
or $n$-dimensional problems with radial symmetry. In \cite{cu}, we have also shown how to recover superexponential stability for a class of degenerate parabolic operators, still applying the above abstract result.

In this paper, we address the related, more delicate, issue of the exact controllability of \eqref{u} to the ground state solution $\psi_1$ via a bilinear control. Such a property, that is obviously stronger than superexponential stabilizability, holds true in more restrictive settings than those considered in \cite{acu}. Nevertheless, our new results, that we state below,  apply to all the aforementioned examples of parabolic problems.
\begin{thm}\label{teo1}
Let $A:D(A)\subset X\to X$ be a densely defined linear operator such that
\begin{equation}\label{ipA}
\begin{array}{ll}
(a) & A \mbox{ is self-adjoint},\\
(b) & A \mbox{ is accretive: }\langle Ax,x\rangle \geq 0,\,\, \forall\, x\in D(A),\\
(c) &\exists\,\,\lambda>0,\,\,(\lambda I+A)^{-1}:X\to X \mbox{ is compact},
\end{array}
\end{equation}
and suppose that there exists a constant $\alpha>0$ for which the eigenvalues of $A$ fulfill the gap condition
\begin{equation}\label{gap}
\sqrt{\lambda_{k+1}}-\sqrt{\lambda_k}\geq \alpha,\quad\forall\, k\in \NN^*.
\end{equation}
Let $B: X\to X$ be a bounded linear operator such that there exist $b,q>0$ for which
\begin{equation}\label{ipB}
\begin{array}{l}
\langle B\varphi_1,\varphi_1\rangle\neq0,\quad\mbox{and}\quad\lambda_k^q|\langle B\varphi_1,\varphi_k\rangle|\geq b\quad\forall\, k>1.
\end{array}
\end{equation}
Then, for any $T>0$, there exists a constant $R_{T}>0$ such that, for any $u_0\in B_{R_{T}}(\varphi_1)$, there exists a control $p\in L^2(0,T)$ for which system \eqref{u} is controllable to the ground state solution in time $T$. 
Furthermore, the following estimate holds
\begin{equation}
||p||_{L^2(0,T)}\leq \frac{e^{-\pi^2C_K/T_f}}{e^{2\pi^2C_K/(3T_f)}-1},
\end{equation}
where
\begin{equation}\label{Talpha}
T_f:=\min\{T,T_\alpha\},\qquad
T_\alpha:=\frac{\pi^2}{6}\min\left\{1,1/\alpha^2\right\}
\end{equation}
and $C_K$ is a suitable positive constant.
\end{thm}

The main idea of the proof consists of applying the stability estimates of \cite{acu} on a suitable sequence of time intervals of decreasing length $T_j$, such that $\sum_{j=1}^\infty T_j<\infty$. Such a sequence, however, has to be suitably chosen in order to fit the error estimates that we take from \cite{acu}.

From the above local exact controllability property we deduce two global controllability results. In the first one, Theorem \ref{teoglobal} below, we prove that all initial states lying in a suitable strip, i.e., satisfying $|\langle u_0,\varphi_1\rangle-1| < r_1$, can be steered to the ground state solution (see Figure~\ref{fig1}). Moreover, we give a uniform estimate for the controllability time.
\begin{thm}\label{teoglobal}
Let $A$ and $B$ satisfy hypotheses \eqref{ipA}, \eqref{gap}, and \eqref{ipB}. Then there exists a constant $r_1>0$ such that for any $R>0$ there exists $T_{R}>0$ such that for all $u_0\in X$ that satisfy
\begin{equation}\label{ipu0}
\begin{array}{l}
\left|\langle u_0,\varphi_1\rangle-1\right|< r_1,\\\\
\left|\left|u_0-\langle u_0,\varphi_1\rangle\varphi_1\right|\right|\leq R,
\end{array}
\end{equation}
problem \eqref{u} is exactly controllable to the ground state solution $\psi_1(t)=e^{-\lambda_1 t}\varphi_1$ in time $T_{R}$.
\end{thm}\newpage
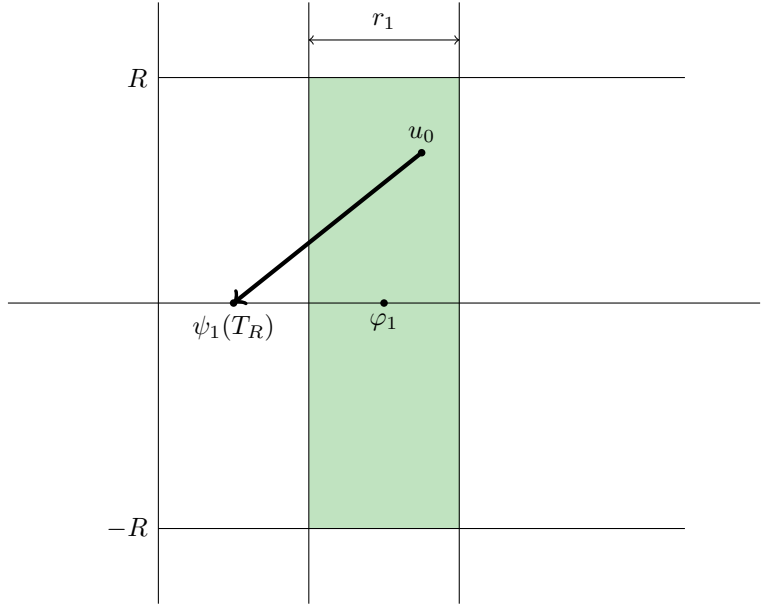
\begin{figure}[ht!]
\centering\begin{tikzpicture}
\fill[green(ryb)!30](4,-3)--(6,-3)--(6,3)--(4,3)--cycle;
\draw[] (0,0)  -- (10,0);
\draw[] (2,4) -- (2,-4);
\fill(5,0) node[below]{\footnotesize{$\varphi_1$}} circle (.05);
\fill(3,0) node[below]{\footnotesize{$\psi_1(T_R)$}} circle (.05);
\draw[] (4,-4) -- (4,4);
\draw[] (6,-4) -- (6,4);
\draw[] (2,3) node[left]{\footnotesize{$R$}} -- (9,3);
\draw[] (2,-3) node[left]{\footnotesize{$-R$}} -- (9,-3);
\draw[<->] (4,3.5) --node[above]{\footnotesize{$r_1$}} (6,3.5);
\fill(5.5,2) node[above]{\footnotesize{$u_0$}} circle (.05);
\draw[ultra thick,->] (5.5,2) -- (3,0);
\end{tikzpicture}
\caption{the colored region represents the set of initial conditions that can be steered to the ground state solution in  time $T_R$.}\label{fig1}
\end{figure}

Our second global result, Theorem \ref{teoglobal0} below, ensures the exact controllability of all initial states $u_0\in X\setminus \varphi_1^\perp$  to the evolution of their orthogonal projection along the ground state solution defined by
\begin{equation}\label{exactphi1}
\phi_1(t)=\langle u_0,\varphi_1\rangle \psi_1(t), \quad\forall\, t \geq 0,
\end{equation}
where $\psi_1$ is the ground state solution.
\begin{thm}\label{teoglobal0}
Let $A$ and $B$ satisfy hypotheses \eqref{ipA}, \eqref{gap} and \eqref{ipB}. 
Then, for any $R>0$ there exists $T_R>0$ such that for all $u_0\in X$ satisfying
\begin{equation}\label{cone}
||u_0-\langle u_0,\varphi_1\rangle\varphi_1||\leq R |\langle u_0,\varphi_1\rangle|
\end{equation}
system \eqref{u} is exactly controllable to $\phi_1$, defined in \eqref{exactphi1}, in time $T_R$.
\end{thm}
Notice that, denoting by $\theta$ the angle between the half-lines $\RR_+\varphi_1$ and $\RR_+ u_0$, condition \eqref{cone} is equivalent to
\begin{equation*}
|\tan\theta|\leq R,
\end{equation*}
which defines a closed cone, say $Q_R$, with vertex at $0$ and axis equal to $\RR\varphi_1$ (see Figure~\ref{fig2}). Therefore, Theorem \ref{teoglobal0} ensures a uniform controllability time for all initial conditions lying in $Q_R$.
We observe that, since $R$ is any arbitrary positive constant, all initial conditions $u_0\in X\setminus \varphi_1^\perp$ can be steered to the corresponding projection to the ground state solution. Indeed, for any $u_0\in X\setminus \varphi_1^\perp$, we define 
\begin{equation*}
R_0:=\left|\left|\frac{u_0}{\langle u_0,\varphi_1\rangle}-\varphi_1\right|\right|.
\end{equation*}
Then, for any $R\geq R_0$ condition \eqref{cone} is fulfilled:
\begin{equation*}
\frac{1}{|\langle u_0,\varphi_1\rangle|}\left|\left|u_0-\langle u_0,\varphi_1\rangle\varphi_1\right|\right|=R_0\leq R.
\end{equation*}
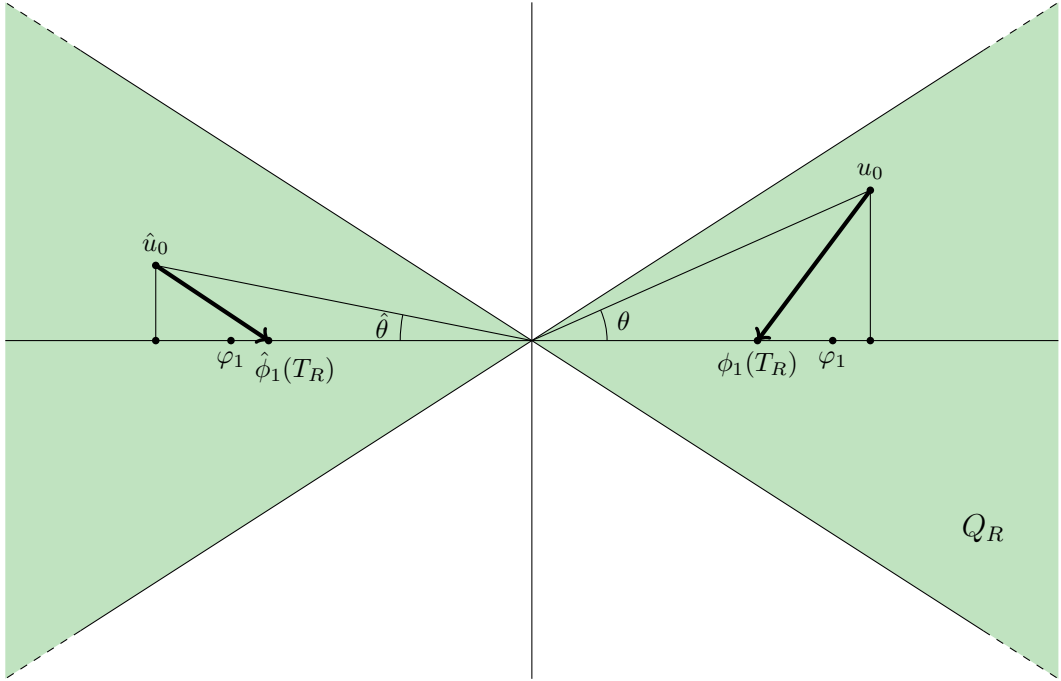
\begin{figure}[ht!]
\centering\begin{tikzpicture}
\fill[green(ryb)!30](0,-4.5)--(7,0)--(0,4.5)--cycle;
\fill[green(ryb)!30](7,0)--(14,-4.5)--(14,4.5)--cycle;
%\coordinate (v3) at (12.5,4);
\coordinate (v2) at (7,0);
%\coordinate (v1) at (12.5,-4);
%\coordinate (v6) at (3.5,4);
%\coordinate (v7) at (3.5,-4);
%\tkzMarkAngle[size=6cm,fill=green(ryb)!30,draw=green(ryb)!30](v1,v2,v3);
%\tkzMarkAngle[size=6cm,fill=green(ryb)!30,draw=green(ryb)!30](v6,v2,v7);
\coordinate(v4) at (11.5,2);
\coordinate(v5) at (11.5,0);
\coordinate(v8) at (2,1);
\coordinate(v9) at (2,0);
\tkzMarkAngle[size=1cm](v5,v2,v4);
\tkzLabelAngle[pos=1.25](v5,v2,v4){\footnotesize{$\theta$}};
\tkzMarkAngle[size=1.75cm](v8,v2,v9);
\tkzLabelAngle[pos=2](v8,v2,v9){\footnotesize{$\hat{\theta}$}};
\draw[] (0,0)  -- (14,0);
\draw[] (7,4.5) -- (7,-4.5);
\draw[] (7,0) -- (13,3.86);
\draw[dashed] (13,3.86) -- (14,4.5);
\draw[] (7,0) -- (13,-3.86);
\draw[dashed] (13,-3.86) -- (14,-4.5);
\draw[] (7,0) -- (1,3.86);
\draw[dashed] (1,3.86) -- (0,4.5);
\draw[] (7,0) -- (1,-3.86);
\draw[dashed] (1,-3.86) -- (0,-4.5);
\fill(11,0) node[below]{\footnotesize{$\varphi_1$}} circle (.05);
\fill(3,0) node[below]{\footnotesize{$\varphi_1$}} circle (.05);
\fill(10,0) node[below]{\footnotesize{$\phi_1(T_R)$}} circle (.05);
\fill(3.5,0) node[below]{\footnotesize{$\qquad\hat{\phi}_1(T_R)$}} circle (.05);
\fill(11.5,2) node[above]{\footnotesize{$u_0$}} circle (.05);
\fill(2,1) node[above]{\footnotesize{$\hat{u}_0$}} circle (.05);
\draw[] (11.5,2) -- (11.5,0);
\draw[] (11.5,2) -- (7,0);
\draw[] (2,1) -- (2,0);
\draw[] (2,1) -- (7,0);
\fill(11.5,0) circle (.05);
\fill(2,0) circle (.05);
\draw[ultra thick,->] (11.5,2) -- (10,0);
\draw[ultra thick,->] (2,1) -- (3.5,0);
\fill(13,-2.5) node{$Q_R$};
\end{tikzpicture}
\caption{fixed any $R>0$, the set of initial conditions exactly controllable in time $T_R$ to their projection along the ground state solution is indicated by the colored cone $Q_R$.}\label{fig2}
\end{figure}

Finally, we would like to recall part of the huge literature on bilinear control of evolution equations, referring the reader to the references in \cite{acu} for more details. A seminal paper in this field is certainly the one by Ball J.M., Marsden J.E., Slemrod M. \cite{bms}, which establishes that system \eqref{u} is not controllable. More precisely, denoting by  $u(t;p,u_0)$  the unique solution of \eqref{u},  the attainable set from $u_0$ defined by
\begin{equation*}
S(u_0)=\{ u(t;p,u_0);t\geq 0, p\in L^r_{loc}([0,+\infty),\RR),r>1\}
\end{equation*}
is shown in \cite{bms} to have a dense complement.

Among positive results, we would like to mention Beauchard K. \cite{b} for bilinear control of the wave equation and Beauchard K., Laurent C. \cite{bl} for the bilinear control of the Schr{\"o}dinger equation (see also \cite{beau} for a first result on this topic). The local exact controllability results obtained in these papers rely on linearization around the ground state, the use of the inverse mapping Theorem, and a regularizing effect which takes place in both problems. The local exact controllability is proved for any positive time for the Schr{\"o}dinger equation and for a sufficiently (optimal) large time for the wave equation. Both papers require the condition 
\begin{equation}
\label{B}
\langle B\varphi_1,\varphi_k\rangle \neq 0, \forall \, k \geq 1
\end{equation}
to be satisfied, together with a suitable asymptotic behavior with respect to the eigenvalues.

If \eqref{B} is violated then it has been first shown by Coron J.-M. \cite{cor}, for a model describing a particle in a moving box, that there exists a minimal time for local exact controllability to hold. This model couples the Schr\"odinger equation with two ordinary differential equations modeling the speed and acceleration of the box (see also Beauchard K., Coron J.-M. \cite{beaucor} for local exact controllability for large time). A further paper by Beauchard K., Morancey M. \cite{bm} for the Schr\"odinger equation extends \cite{bl} to cases for which the above condition is violated, namely there exist integers $k$ for which 
$\langle B\varphi_1,\varphi_k\rangle =0$. Under some additional assumptions, the authors prove that the system is locally exactly controllable in a sufficiently large time.

%Unlike the above references, the present paper addresses abstract evolution equations of  parabolic type. Theorem 1.1 in our paper uses part of the strategy of \cite{bl} and properties specific to parabolic equations to prove local exact controllability for arbitrary positive time (the underlying semigroup generated by the linear part of the equation is analytic). We also prove global exact controllability results (see Theorem 1.2 and 1.3) due to specific properties of parabolic equations (the linear part of the equation is exponentially dissipative).

Another example of controllability to trajectories for nonlinear parabolic systems is studied in \cite{fgip}, where, however, additive controls are considered. In such an example, one can obtain controllability to free trajectories by Carleman estimate and inverse mapping arguments. Such a strategy seems hard to adapt to the current setting.

It is worth noting that the bilinear  controls we use in this paper are just scalar functions of time. This fact explains why applications mainly concern problems in low space dimension. 
A stronger control action could be obtained  by letting controls depend on time and space. We refer the reader to \cite{ck,cfk} for more on this subject.

This paper is organized as follows. In section 2, we have collected some preliminaries as well as results from \cite{acu} that we need for the proof of Theorem \ref{teo1}. Section 3 contains such a proof, while section 4 is devoted to examples of applications to parabolic problems.
\section{Preliminaries}
In this section we recall some results from \cite{acu} that are necessary for the construction of the proof of Theorem \ref{teo1}. First, fixed $T>0$, consider the following bilinear control problem
\begin{equation}\label{a1f}\left\{
\begin{array}{ll}
u'(t)+A u(t)+p(t)Bu(t)+f(t)=0,& t\in [0,T]\\\\
u(0)=u_0.
\end{array}\right.
\end{equation}
We introduce the following notation: 
\begin{equation*}\begin{array}{l}
||f||_{2,0}:=||f||_{L^2(0,T;X)},\qquad\forall\,f\in L^2(0,T;X)\\\\
||f||_{\infty,0}:=||f||_{C([0,T];X)}=\sup_{t\in [0,T]}||f(t)||,\qquad\forall\, f\in C([0,T];X).
\end{array}
\end{equation*}
The well-posedness of \eqref{a1f} is ensured by the following Proposition.
\begin{prop}\label{propa24}
Let $T>0$. If $u_0\in X$, $p\in L^2(0,T)$ and $f\in L^2(0,T;X)$, then there exists a unique mild solution of \eqref{a1f}, i.e. a function $u\in C([0,T];X)$ such that the following equality holds for every $t\in [0,T]$,

\begin{equation}
u(t)=e^{-tA }u_0-\int_0^t e^{-(t-s)A}[p(s)Bu(s)+f(s)]ds.
\end{equation}
Moreover, there exists a constant $C_1(T)>0$ such that 
\begin{equation}\label{a5}
||u||_{\infty,0}\leq C_1(T) (||u_0||+||f||_{2,0}).
\end{equation}
\end{prop}
\noindent We refer to \cite{acu} for a proof of Proposition \ref{propa24}.

Our aim is to show the controllability of the following system
\begin{equation}\label{sys}\left\{\begin{array}{ll}
u'(t)+A u(t)+p(t)Bu(t)=0,& t\in[0,T]\\\\
u(0)=u_0,
\end{array}\right.
\end{equation}
to the ground state solution $\psi_1=e^{-\lambda_1 t}\varphi_1$, that is the solution of \eqref{sys} when $p=0$ and $u_0=\varphi_1$. We first consider the case $\lambda_1=0$ and prove the controllability result to the corresponding ground state solution $\psi_1=\varphi_1$. Then, we recover the result also for the case $\lambda_1>0$.

Set $v:=u-\varphi_1$, then $v$ is the solution of the following Cauchy problem
\begin{equation}\label{v}
\left\{\begin{array}{ll}
v'(t)+A v(t)+p(t)Bv(t)+p(t)B\varphi_1=0,&t\in[0,T]\\\\
v(0)=v_0=u_0-\varphi_1.
\end{array}\right.
\end{equation}
We observe that the controllability of $u$ to $\varphi_1$ is equivalent to the null controllabiliy of \eqref{v}. In order to prove this latter result, we consider the following linearized system
\begin{equation}\label{lin}\left\{\begin{array}{ll}
\bar{v}(t)'+A \bar{v}(t)+p(t)B\varphi_1=0,&t\in[0,T]\\\\
\bar{v}(0)=v_0.
\end{array}\right.
\end{equation}
and we introduce the following constant
\begin{equation}\label{lambdaT}
\Lambda_T:=\left( \sum_{k\in\NN^*}\frac{e^{-2\lambda_kT}e^{\bar{C}\sqrt{\lambda_k}/\alpha}}{|\langle B\varphi_1,\varphi_k\rangle|^2}\right)^{1/2}.
\end{equation}
where $\alpha$ is the constant in \eqref{gap}. We observe that, thanks to assumption \eqref{ipB}, $\Lambda_T$ converges for any $T>0$.
The next Proposition guarantees the null controllability of \eqref{lin} and furthermore it yields an estimate of the control $p$ with respect to the initial condition $v_0$.
\begin{prop}\label{prop34}
Let $T>0$ and let $A$ and $B$ be such that \eqref{ipA}, \eqref{gap}, \eqref{ipB} hold. Furthermore, assume that $\lambda_1=0$. Let $v_0\in X$. Then, defining $p\in L^2(0,T)$ by
\begin{equation}\label{pdef}
p(t)=\sum_{k\in\NN^*}\frac{\langle v_0,\varphi_k\rangle}{\langle B\varphi_1,\varphi_k\rangle}\sigma_k(t)
\end{equation}
where $\{\sigma_k\}_{k\in\NN^*}$ is the biorthogonal family to $\{e^{\lambda_kt}\}_{k\in\NN^*}$ given by \cite[Theorem 2.4]{cmv}, it holds that $\bar{v}(T)=0$.

Moreover, there exists a constant $C_\alpha(T)>0$ such that
\begin{equation}\label{pbound}
||p||_{L^2(0,T)}\leq C_\alpha(T)\Lambda_T||v_0||
\end{equation}
where $\Lambda_T$ is defined in \eqref{lambdaT} and $\alpha$ is the constant in \eqref{gap}.
\end{prop}
\noindent For the proof of Proposition \ref{prop34} we refer to \cite{acu}.

\begin{oss}
The behavior of $C_{\alpha}(\cdot)$ with respect to its argument has been studied in \cite{cmv} and is given by
\begin{equation}
C_{\alpha}^2(T)=\bar{C}\cdot\left\{\begin{array}{ll}\left(\frac{1}{T}+\frac{1}{T^2\alpha^2}\right)e^{\bar{C}/(T\alpha^2)},& T\leq\frac{1}{\alpha^2}\\\\
\bar{C}\alpha^2,&T\geq \frac{1}{\alpha^2}, \end{array}\right.
\end{equation}
where $\bar{C}>0$ is a constant independent of $T$ and $\alpha$.
\end{oss}
We now use the control $p$ built in Proposition \ref{prop34} also in the nonlinear system \eqref{v} and we give an estimate for the corresponding solution $v$.
\begin{prop}\label{prop38}
Let $A$ and $B$ satisfy hypotheses \eqref{ipA}, \eqref{gap}, \eqref{ipB} and, furthermore, assume $\lambda_1=0$. Let $p\in L^2(0,T)$ be defined by \eqref{pdef}. Then, the solution $v$ of \eqref{v} satisfies
\begin{equation}\label{unifv}
\sup_{t\in[0,T]}||v(t)||^2\leq e^{C_3(T)\Lambda_T||v_0||+C_BT}(1+C_4(T)\Lambda_T^2)||v_0||^2
\end{equation}
where $C_B\geq1$ is the norm of the operator $B$, $C_3(T):=2\sqrt{T}C_BC_\alpha(T)$, and $C_4(T):=C_BC_\alpha^2(T)$.
\end{prop}
\noindent For the proof of Proposition \ref{prop38} we refer to \cite{}.

We introduce the function $w(t):=v(t)-\bar{v}(t)$ that satisfies the following Cauchy problem
\begin{equation}\label{w}
\left\{\begin{array}{ll}
w'(t)+Aw(t)+p(t)Bv(t)=0,&t\in[0,T]\\\\
w(0)=0.
\end{array}\right.
\end{equation}
We define the function $K$ on $(0,\infty)$ by
\begin{equation}\label{KT}
K^2(T):=C_Be^{2C_B\sqrt{T}+(C_B+1)T}C_4(T)\Lambda^2_T(1+C_4(T)\Lambda_T^2).
\end{equation}
In the following Proposition we estimate how close we are able to steer $v$ to $0$ in time $T$ by means of the control $p$ defined in \eqref{pdef}.
\begin{prop}\label{prop39}
Let $A$ and $B$ satisfy hypotheses \eqref{ipA}, \eqref{gap}, \eqref{ipB}, and, furthermore, we assume $\lambda_1=0$. Let $T>0$, $p$ be defined by \eqref{pdef}, and let $v_0\in X$ be such that 
\begin{equation}\label{v0}
C_\alpha(T)\Lambda_T||v_0||\leq 1.
\end{equation}
Then, it holds that
\begin{equation}\label{wT}
||w(T)||=||v(T)||\leq K(T)||v_0||^2.
\end{equation}
\end{prop}

\begin{proof}
Observe that $w\in C([0,T];X)$ is the mild solution of \eqref{w}. Moreover $w\in H^1(0,T;X)\cap L^2(0,T;D(A))$ and thus $w$ satisfies the equality
\begin{equation}\label{eqw}
w'(t)+Aw(t)+p(t)Bv(t)=0
\end{equation}
for almost every $t\in [0,T]$.

We multiply equation \eqref{eqw} by $w(t)$ and we obtain
\begin{equation}\begin{split}
\frac{1}{2}\frac{d}{dt}||w(t)||^2&\leq |p(t)|||Bv(t)||||w(t)||\\
&\leq \frac{1}{2}||w(t)||^2+C^2_B\frac{1}{2}|p(t)|^2||v(t)||^2.
\end{split}
\end{equation}
Therefore, applying Gronwall's inequality, taking the supremum over $[0,T]$ and using \eqref{unifv} and \eqref{pbound}, we get
\begin{equation}\begin{split}
\sup_{t\in[0,T]}||w(t)||^2&\leq C^2_Be^T||p||^2_{L^2(0,T)}\sup_{t\in[0,T]}||v(t)||^2\\
&\leq C^2_Be^{C_3(T)\Lambda_T||v_0||+C_BT+T}(1+C_4(T)\Lambda_T^2)||v_0||^2||p||^2_{L^2(0,T)}\\
&\leq C^2_BC^2_\alpha(T)\Lambda_T^2e^{C_3(T)\Lambda_T||v_0||+(C_B+1)T}(1+C_4(T)\Lambda_T^2)||v_0||^4.
\end{split}
\end{equation}
We can suppose, without loss of generality, that $C_\alpha(T)\geq 1$. Thus, thanks to \eqref{v0}, we obtain 
\begin{equation*}
\sup_{t\in[0,T]}||w(t)||^2\leq C^2_BC^2_\alpha(T)\Lambda_T^2e^{2C_B\sqrt{T}+(C_B+1)T}(1+C_4(T)\Lambda_T^2)||v_0||^4
\end{equation*}
that is equivalent to
\begin{equation*}
\sup_{t\in[0,T]}||w(t)||^2\leq K(T)^2||v_0||^4.
\end{equation*}
By the last inequality we infer that
\begin{equation}\label{vT}
||w(T)||\leq K(T)||v_0||^2.
\end{equation}
\end{proof}

\section{Proof of Theorem \ref{teo1}}
Fixed $0<T\leq\min\left\{1,1/\alpha^2\right\}$, we define the sequence $\{T_j\}_{j\in\NN^*}$ by
\begin{equation}\label{Tj}
T_j:=T/j^2,
\end{equation}
and the time steps
\begin{equation}\label{taun}
\tau_n=\sum_{j=1}^n T_j,\qquad\forall n\in\NN,
\end{equation}
with the convention that $\sum_{j=1}^0T_j=0$. Notice that $\sum_{j=1}^\infty T_j=\frac{\pi^2}{6}T$.

The proof of our result relies on the construction of the solution $v$ of \eqref{v} in consecutive intervals of the form $\left[\tau_n,\tau_{n+1}\right]$ for which we are able to perform an iterate estimate of \eqref{wT}.

First, through the following Lemma, we study the behavior of the constant $K(T)$ with respect to $T$.

We define the function
\begin{equation}\label{G}
G_M(T):=\frac{M}{T^2}e^{M/T}\sum_{k=1}^{\infty}\frac{e^{-2\lambda_kT+M\sqrt{\lambda_k}}}{|\langle B\varphi_1,\varphi_k\rangle|^2},\qquad 0<T\leq 1
\end{equation}
where $M$ is a positive constant.
\begin{lem}\label{lem1}
Let $A:D(A)\subset X\to X$ be such that \eqref{ipA} and \eqref{gap} hold and $B:X\to X$ be such that \eqref{ipB} holds. Then, there exists a suitable positive constant $C_M$ such that
\begin{equation}
G_M(T)\leq e^{C_M/T}, \qquad \forall\,\, 0<T\leq 1.
\end{equation}
\end{lem}
\begin{proof}
Thanks to assumption \eqref{ipB}, we have that
\begin{equation}\
\begin{split}
G_M(T)&=\frac{M}{T^2}e^{M/T}\sum_{k=1}^{\infty}\frac{e^{-2\lambda_kT+M\sqrt{\lambda_k}}}{|\langle B\varphi_1,\varphi_k\rangle|^2}\\
&\leq \frac{M}{T^2}e^{M/T}\left[\frac{e^{M^2/(8T)}}{|\langle B\varphi_1,\varphi_1\rangle|^2}+\frac{1}{b^2}\sum_{k=2}^{\infty}\left(\lambda_k^{2q}e^{-\lambda_kT}\right)e^{-\lambda_kT+M\sqrt{\lambda_k}}\right].
\end{split}
\end{equation}
For any $\lambda\geq0$ we set $f(\lambda)=e^{-\lambda T+M\sqrt{\lambda}}$. The maximum value of $f$ is attained at $\lambda=\left(\frac{M}{2T}\right)^2$. So, we can bound $G_M(T)$ as follows
\begin{equation}\label{gamma4}
G_M(T)\leq\frac{M}{T^2}e^{M/T}\left[\frac{e^{M^2/(8T)}}{|\langle B\varphi_1,\varphi_1\rangle|^2}+\frac{e^{M^2/(4T)}}{b^2}\sum_{k=1}^{\infty}\lambda_k^{2q}e^{-\lambda_kT}\right].
\end{equation}

Now, for any $\lambda\geq0$ we define the function $g(\lambda)=\lambda^{2q}e^{-\lambda T}$. Its derivative is given by
\begin{equation*}
g'(\lambda)=(2q-\lambda T)\lambda^{2q-1}e^{-\lambda T}
\end{equation*}
and therefore we deduce that
\begin{equation*}
g(\lambda)\mbox{ is }\left\{\begin{array}{ll}\mbox{increasing } & \mbox{if }0\leq \lambda<(2q)/T\\\\
\mbox{decreasing}&\mbox{if } \lambda\geq (2q)/T  \end{array}\right.
\end{equation*}
and $g$ has a maximum at $\lambda=(2q)/T$. We define the following index:
\begin{equation*}
k_1:=k_1(T)=\sup\left\{k\in\NN^*\,:\,\lambda_k\leq \frac{2q}{T}\right\}
\end{equation*}
Note that $k_1(T)$ goes to $\infty$ as $T$ converges to $0$.
We can rewrite the sum in \eqref{gamma4} as follows
\begin{equation}\label{s123}
\sum_{k=1}^{\infty}\lambda_k^{2q}e^{-\lambda_kT}=\sum_{k\leq k_1-1}\lambda_k^{2q}e^{-\lambda_kT}+\sum_{k_1\leq k\leq k_1+1}\lambda_k^{2q}e^{-\lambda_kT}+\sum_{k\geq k_1+2}\lambda_k^{2q}e^{-\lambda_kT}.
\end{equation}

For any $k\leq k_1-1$, we have
\begin{equation}\label{s1}
\int_{\lambda_k}^{\lambda_{k+1}}\lambda^{2q}e^{-\lambda T}d\lambda\geq (\lambda_{k+1}-\lambda_k)\lambda_k^{2q}e^{-\lambda_k T}\geq \alpha(\sqrt{\lambda_2}+\sqrt{\lambda_1})\lambda_k^{2q}e^{-\lambda_k T}
\end{equation}
and for any $k\geq k_1+2$
\begin{equation}\label{s3}
\int_{\lambda_{k-1}}^{\lambda_k}\lambda^{2q}e^{-\lambda T}d\lambda\geq (\lambda_k-\lambda_{k-1})\lambda_k^{2q}e^{-\lambda_k T}\geq \alpha(\sqrt{\lambda_2}+\sqrt{\lambda_1})\lambda_k^{2q}e^{-\lambda_k T}.
\end{equation}
So, by using estimates \eqref{s1} and \eqref{s3}, \eqref{s123} becomes
\begin{equation}\label{s2i}
\sum_{k=1}^{\infty}\lambda_k^{2q}e^{-\lambda_kT}=\frac{2}{\alpha(\sqrt{\lambda_2}+\sqrt{\lambda_1})}\int_0^\infty \lambda^{2q}e^{-\lambda T}d\lambda+\sum_{k_1\leq k\leq k_1+1}\lambda_k^{2q}e^{-\lambda_kT}.
\end{equation}
Furthermore, recalling that $g$ has a maximum for $\lambda=2q/T$, it holds that
\begin{equation}\label{s2}
k=k_1,k_1+1\quad\Rightarrow\quad \lambda_k^{2q}e^{-\lambda_k T}\leq\left(2q/T\right)^{2q}e^{-2q}.
\end{equation}

Finally, the integral term of \eqref{s2i} can be rewritten as
\begin{equation}\label{s13}
\int_0^\infty\lambda^{2q}e^{-\lambda T}d\lambda=\frac{1}{T}\int_0^{\infty}\left(\frac{s}{T}\right)^{2q}e^{-s}ds=\frac{1}{T^{1+2q}}\int_0^{\infty}s^{2q}e^{-s}ds=\frac{\Gamma(2q+1)}{T^{1+2q}},
\end{equation}
where by $\Gamma(\cdot)$ we indicate the Euler integral of the second kind. 

Therefore, we conclude from \eqref{s2} and \eqref{s13} that there exist two constants $C_q,C_{q,\alpha}>0$ such that 
\begin{equation}
\sum_{k=1}^\infty \lambda_k^{2q}e^{-\lambda_k T}\leq \frac{C_q}{T^{2q}}+\frac{C_{\alpha,q}}{T^{1+2q}}.
\end{equation}
We use this last bound to prove that there exists $C_M>0$ such that
\begin{equation*}
G_M(T)\leq\frac{M}{T^2}e^{M/T}\left[\frac{e^{M^2/(8T)}}{|\langle B\varphi_1,\varphi_1\rangle|^2}+\frac{e^{M^2/(4T)}}{b^2}\left(\frac{C_q}{T^{2q}}+\frac{C_{\alpha,q}}{T^{1+2q}}\right)\right]\leq e^{C_M/T}, \qquad 0<T\leq 1
\end{equation*}
as claimed.
%Basta prendere C_G=C+C+3+2p
\end{proof}

\begin{oss}\label{oss32}
We recall that $K(\cdot)$ is defined by
$$K^2(T):=C^2_Be^{2C_B\sqrt{T}+(C_B+1)T}C^2_\alpha(T)\Lambda^2_T(1+C_BC^2_\alpha(T)\Lambda_T^2).$$
For any $0<T\leq\min\left\{1,1/\alpha^2\right\}$, $C^2_\alpha(\cdot)$ is given by
$$C_\alpha^2(T)=\bar{C}\left(\frac{1}{T}+\frac{1}{T^2\alpha^2}\right)e^{\bar{C}/(\alpha^2 T)}.$$
Thus, we have the following bound for $K(\cdot)$
\begin{equation}
K(T)^2\leq C^2_Be^{2C_B\sqrt{T}+(C_B+1)T}G_M(T)\left(1+C_BG_M(T)\right),
\end{equation}
where $G_M(\cdot)$ is defined by \eqref{G} and the subscribed $M$ is given by $M=\bar{C}\left(1+\frac{1}{\alpha^2}\right)$.

Thanks to Lemma \ref{lem1}, we infer that there exists a suitable constant $C_K>0$ such that $C_K>C_M$ and 
\begin{equation}
K(T)\leq e^{C_K/T}, \qquad\forall\, T\in(0,1].
\end{equation}
\end{oss}

In the following Proposition we prove that it is possible to iterate the construction of $v$ in consecutive time intervals of the form $[\tau_{n-1},\tau_n]$.

\begin{prop}\label{prop33}
Let $0<T\leq\min\left\{1,1/\alpha^2\right\}$, and consider the sequence $(T_j)_{j\in\NN^*}$ defined by \eqref{Tj}.  Let $v_0\in X$ for which $||v_0||< e^{-6C_K/T}$, let $A:D(A)\subset X\to X$ be such that \eqref{ipA} and \eqref{gap} hold and let $B:X\to X$ satisfies \eqref{ipB}.  Moreover, we assume that $\lambda_1=0$. Then, for every $n\in\NN^*$, problem
\begin{equation}\label{exact44}
\left\{\begin{array}{ll}
v'(t)+Av(t)+p(t)Bv(t)+p(t)B\varphi_1=0,&t\in \left[\tau_{n-1},\tau_n\right]\\\\
v\left(\tau_{n-1}\right)=v_{n-1},
\end{array}
\right.
\end{equation}
where $v_{n-1}$ is determined by induction from the previous steps, $p\in L^2(\tau_{n-1},\tau_n)$ is given by
\begin{equation}\label{pn}
p(t)=\sum_{k=1}^\infty\frac{\langle v_{n-1},\varphi_k\rangle}{\langle B\varphi_1,\varphi_k\rangle}\sigma_k(t-\tau_{n-1}),
\end{equation}
admits a unique mild solution $v\in C\left(\left[\tau_{n-1},\tau_n\right],X\right)$ that satisfies
\begin{equation}
\left|\left| v\left(\tau_n\right)\right|\right|\leq e^{\left(\sum_{j=1}^n 2^{n-j}j^2-2^n6\right)C_K/T},
\end{equation}
where the time steps $\{\tau_n\}_{n\in\NN}$ are defined by \eqref{taun}.
\end{prop}

\begin{proof}
To prove the result, we proceed by induction on $n$. For $n=1$, by Proposition \ref{prop39}, the hypothesis on $v_0$ and Remark \ref{oss32}, $v$ it satisfies
\begin{equation*}
||v(T)||\leq K(T)||v_0||^2\leq e^{-11C_K/T}.
\end{equation*}
Now, suppose the statement is true for all indices $k\leq n-1$, we show the validity for index $n$. Therefore, by inductive hypothesis, the solution $v$ has been constructed in consecutive intervals until $\left[\tau_{n-2},\tau_{n-1}\right]$ and it satisfies
\begin{equation*}
\left|\left| v\left(\tau_{n-1}\right)\right|\right|\leq e^{\left(\sum_{j=1}^{n-1} 2^{n-1-j}j^2-2^{n-1}6\right)C_K/T}.
\end{equation*}
Hence,
\begin{equation}\label{induc}
\begin{split}
C_\alpha(T_n)\Lambda_{T_n}\left|\left| v\left(\tau_{n-1}\right)\right|\right|&\leq e^{C_Mn^2/T}e^{\left(\sum_{j=1}^{n-1} 2^{n-1-j}j^2-2^{n-1}6\right)C_K/T}\\
&\leq e^{(n^2+(-(n-1)^2-4(n-1)+2^{n-1}6-6-2^{n-1}6)C_K/T}\\
&=e^{-(2n+3)C_K/T},
\end{split}
\end{equation}
where we have used that $C_M<C_K$ and the identity
\begin{equation}
\sum_{j=0}^n\frac{j^2}{2^j}=2^{-n}(-n^2-4n+6(2^n-1)), \qquad n\geq0,
\end{equation}
which can be easily checked by induction.

Consider problem \eqref{exact44} with $v_{n-1}$ the solution built in the previous interval, evaluated at $\tau_{n-1}$. By the change of variables $s=t-\tau_{n-1}$, we shift \eqref{exact44} into the interval $[0,T_n]$. We introduce the functions $\tilde{v}(s)=v\left(s+\tau_{n-1}\right)$ and $\tilde{p}(s)=p\left(s+\tau_{n-1}\right)$ and we rewrite \eqref{exact44} as
\begin{equation}\label{tildevtn-1tn}
\left\{\begin{array}{ll}
\tilde{v}'(s)+A\tilde{v}(s)+\tilde{p}(s)B\tilde{v}(s)+\tilde{p}(s)B\varphi_1=0,&s\in \left[0,T_n\right]\\\\
\tilde{v}(0)=v_{n-1}.
\end{array}
\right.
\end{equation}
From \eqref{induc} it follows that $C_\alpha(T_n)\Lambda_{T_n}\left|\left| v(\tau_{n-1})\right|\right|<1$ and thus we can apply Proposition \ref{prop39} to problem \eqref{tildevtn-1tn}, obtaining
\begin{equation}
||\tilde{v}(T_n)||\leq K(T_n)||v_{n-1}||^2.
\end{equation}
We shift back the problem into the original interval $\left[\tau_{n-1},\tau_{n}\right]$ and we get
\begin{equation}
\left|\left| v\left(\tau_{n}\right)\right|\right|\leq K(T_n)||v_{n-1}||^2.
\end{equation}
By inductive hypothesis, we can estimate $\left|\left| v\left(\tau_{n}\right)\right|\right|$ as follows
\begin{equation}
\left|\left| v\left(\tau_{n}\right)\right|\right|\leq e^{C_Kn^2/T}\left[e^{\left(\sum_{j=1}^{n-1} 2^{n-1-j}j^2-2^{n-1}6\right)C_K/T}\right]^2=e^{\left(\sum_{j=1}^n 2^{n-j}j^2-2^n6\right)C_K/T}.
\end{equation}
\end{proof}

\begin{prop}
Let $0<T\leq\min\left\{1,1/\alpha^2\right\}$ and consider the sequence $(T_j)_{j\in\NN^*}$  defined by \eqref{Tj}. Let $v_0\in X$ be such that $||v_0||< e^{-6C_K/T}$, let $A:D(A)\subset X\to X$ be such that \eqref{ipA} and \eqref{gap} hold and let $B:X\to X$ satisfies \eqref{ipB}. Let $p\in L^2(\tau_{n-1},\tau_n)$ be defined by \eqref{pn}.  Moreover, we assume that $\lambda_1=0$. Then, the solution of \eqref{exact44} satisfies
\begin{equation}\label{estimvn}
\left|\left|v\left(\tau_n\right)\right|\right|\leq \prod_{j=1}^nK(T_j)^{2^{n-j}}||v_0||^{2^n},
\end{equation}
for all $n\in\NN^*$.
\end{prop}
\begin{proof}
We prove formula \eqref{estimvn} by induction on $n$. The case $n=1$ follows from Proposition \ref{prop39}, thanks to the assumption $||v_0||< e^{-6C_K/T}$. Now, suppose the formula holds for all the indices less than or equal to $n-1$. We prove it for $n$ as follows. We consider problem \eqref{exact44} and in order to shift it in the interval $[0,T_n]$, we introduce the variable $s=t-\tau_{n-1}$ as before and the functions $\tilde{v}(s)=v\left(s+\tau_{n-1}\right)$ and $\tilde{p}(s)=p\left(s+\tau_{n-1}\right)$. Thus, \eqref{exact44} can be rewritten as
\begin{equation}\label{vns}
\left\{\begin{array}{ll}
\tilde{v}'(s)+A\tilde{v}(s)+\tilde{p}(s)B\tilde{v}(s)+\tilde{p}(s)B\varphi_1=0,&s\in \left[0,T_n\right]\\\\
\tilde{v}(0)=v_{n-1}.
\end{array}
\right.
\end{equation}
By Proposition \ref{prop33}, it holds that $C_\alpha(T_n)\Lambda_{T_n}||v_{n-1}||\leq 1$
and hence, we can apply Proposition \ref{prop39} considering as final time $T_n$ (instead of $T$), obtaining that
\begin{equation}
\left|\left|v\left(\tau_{n}\right)\right|\right|=||\tilde{v}(T_n)||\leq K(T_n)||v_{n-1}||^2.
\end{equation}
Finally, by inductive hypothesis, we conclude that
\begin{equation}
\left|\left|v\left(\tau_{n}\right)\right|\right|\leq K(T_n)||v_{n-1}||^2\leq K(T_n)\left[\prod_{j=1}^{n-1}K(T_j)^{2^{n-1-j}}||v_0||^{2^{n-1}}\right]^2
\end{equation}
that is equivalent to formula \eqref{estimvn}.
\end{proof}

We are now ready to prove our main result. 
\begin{proof}[Proof of Theorem \ref{teo1}]
We start the proof by considering the case in which $\lambda_1=0$. Let $T>0$ and let $T_\alpha$ and $T_f$ be defined by \eqref{Talpha}. We define $\tilde{T}=\frac{6}{\pi^2}T_f$ and $R_{T}:=e^{-\pi^2C_K/T_f}$. Observe that $0<\tilde{T}\leq1$ and we define the time steps $\{\tau_n\}_{n\in\NN}$ as in \eqref{taun} with $T_j:=\tilde{T}/j^2$. Fixed $v_0\in B_{R_{T}}(0)$, we apply \eqref{estimvn} to obtain
\begin{equation}\label{final}
\begin{split}
\left|\left|v\left(\tau_{n}\right)\right|\right|&\leq\prod_{j=1}^nK(T_j)^{2^{n-j}}||v_0||^{2^n}\\
&\leq \prod_{j=1}^n\left(e^{C_Kj^2/\tilde{T}}\right)^{2^{n-j}}||v_0||^{2^n}\\
&=e^{C_K2^n/\tilde{T}\sum_{j=1}^nj^2/2^j}||v_0||^{2^n}\\
&\leq e^{C_K2^n/\tilde{T}\sum_{j=1}^\infty j^2/2^j}||v_0||^{2^n}\\
&\leq \left(e^{6C_K/\tilde{T}}||v_0||\right)^{2^n}
\end{split}
\end{equation}
where we have used that $\sum_{j=1}^\infty j^2/2^j=6$.
We take the limit as $n\to \infty$ of \eqref{final} and we get
\begin{equation}
\left|\left|u\left(\frac{\pi^2}{6}\tilde{T}\right)-\varphi_1\right|\right|=\left|\left|v\left(\frac{\pi^2}{6}\tilde{T}\right)\right|\right|=||v(T_f)||\leq 0
\end{equation}
since $||v_0||< e^{-\pi^2C_K/T_f}=e^{-6C_K/\tilde{T}}$. This means that, we have built a control $p\in L^2_{loc}([0,\infty))$, defined by
\begin{equation}
p(t)=\left\{\begin{array}{ll}
\sum_{n=0}^\infty p_n(t)\chi_{\left[\tau_n ,\tau_{n+1}\right]}(t),& t\in \left(0,T_f\right],\\\\
0,&t\in(T_f,+\infty)
\end{array}\right.
\end{equation}
where
\begin{equation}
p_n(t)=\sum_{k=1}^\infty \frac{\langle v\left(\tau_n\right),\varphi_k\rangle}{\langle B\varphi_1,\varphi_k\rangle}\sigma_k\left( t-\tau_n \right),\qquad\forall n\in\NN,
\end{equation} 
such that the solution $u$ of \eqref{u} reaches the ground state solution $\varphi_1$ in time $T$, and stays on it forever.

Observe that, thanks to \eqref{pbound} and \eqref{induc}, we are able to yield a bound for the $L^2$-norm of such a control:
\begin{equation}\label{pestimate}
\begin{split}
||p||^2_{L^2\left(0,T\right)}&=\sum_{n=0}^\infty ||p_n||^2_{L^2\left(\tau_n,\tau_{n+1}\right)}\\
&\leq \sum_{n=0}^\infty \left(C_{\alpha}(T_{n+1})\Lambda_{T_{n+1}}\left|\left|v\left(\tau_n \right)\right|\right|\right)^2\\
&\leq \sum_{n=0}^\infty e^{-2(2(n+1)+3)C_K/\tilde{T}}\\
&=\frac{e^{-6C_K/\tilde{T}}}{e^{4C_K/\tilde{T}}-1}\\
&=\frac{e^{-\pi^2C_K/T_f}}{e^{2\pi^2C_K/(3T_f)}-1}
\end{split}
\end{equation}

Now we face the case $\lambda_1>0$. We define the operator
\begin{equation*}
A_1:=A-\lambda_1I.
\end{equation*}
It is possible to check that $A_1$ satisfies \eqref{ipA} and moreover it has the same eigenfuctions, $\{\varphi_k\}_{k\in\NN^*}$, of $A$, while the eigenvalues are given by
\begin{equation}
\mu_k=\lambda_k-\lambda_1, \qquad\forall k\in\NN^*.
\end{equation}
In particular, $\mu_1=0$ and furthermore, $\{\mu_k\}_{k\in\NN^*}$ satisfy the same gap condition \eqref{gap} fulfilled by $\{\lambda_k\}_{k\in\NN^*}$.

We define the function $z(t)=e^{\lambda_1 t}u(t)$, where $u$ is the solution of \eqref{u}. Then, $z$ solves the following problem
\begin{equation}\label{z}
\left\{\begin{array}{ll}
z'(t)+A_1z(t)+p(t)Bz(t)=0,&t>0,\\\\
z(0)=u_0.
\end{array}\right.
\end{equation}
For any $T>0$, we define $T_f$ as in \eqref{Talpha} and the constant $R_{T}:=e^{-\pi^2C_K/T_f}$. We deduce from the previous analysis that, if $u_0\in B_{R_{T}}(\varphi_1)$, then there exists a control $p\in L^2([0,+\infty))$ that steers the solution $z$ to the ground state solution $\varphi_1$ in time $T_f\leq T$. This implies the exact controllability of $u$ to the ground state solution $\psi_1(t)=e^{-\lambda_1 t}\varphi_1$: indeed,
{\footnotesize{\begin{equation*}
\left|\left|u\left(T_f\right)-\psi_1\left(T_f\right)\right|\right|=\left|\left|e^{-\lambda_1T_f}z\left(T_f\right)-e^{-\lambda_1T_f}\varphi_1\right|\right|=e^{-\lambda_1T_f}\left|\left|z\left(T_f\right)-\varphi_1\right|\right|=0.
\end{equation*}}}
This concludes the proof of our Theorem.
\end{proof}
\begin{oss}
We observe that, from \eqref{pestimate}, it follows that $||p||_{L^2(0,T_{f})}\to 0$ as $T_f\to 0$. This fact is not surprising because as $T_f$ approaches $0$, also the size of the neighborhood where the initial condition can be chosen goes to zero.
\end{oss}

\section{Proof of Theorems \ref{teoglobal} and \ref{teoglobal0}}
Before proving Theorem \ref{teoglobal}, let us show a preliminary result that demonstrates the statement in the case of a strictly accretive operator.
\begin{lem}\label{lemglobal}
Let $A$ and $B$ satisfy hypotheses \eqref{ipA}, \eqref{gap} and \eqref{ipB}.  Furthermore, we assume $\lambda_1=0$. Then, there exists a constant $r_1>0$ such that for any $R>0$ there exists $T_{R}>0$ such that for all $v_0\in X$ that satisfy
\begin{equation}\label{ipv0}
\begin{array}{l}
\left|\langle v_0,\varphi_1\rangle\right| <  r_1,\\\\
\left|\left|v_0-\langle v_0,\varphi_1\rangle\varphi_1\right|\right|\leq R,
\end{array}
\end{equation}
problem \eqref{v} is null controllable in time $T_{R}$.
\end{lem}
\begin{proof}
{\bf First step.} We fix $T=1$. Thanks to Theorem \ref{teo1}, there exists a constant $r_1>0$ such that if $||u_1(0)- \varphi_1|| < \sqrt{2}r_1$ then there exists a control $p_1\in L^2(0,1)$ for which the solution $u_1$ of \eqref{u} on $[0,1]$ with $p$ replaced by $p_1$, satisfies $u_1(1)=\varphi_1$. We set $v_1=u_1- \varphi_1$ on $[0,1]$. We deduce that
if $||v_1(0)|| < \sqrt{2}r_1$ then there exists a control $p_1\in L^2(0,1)$ for which the solution $v_1$ of \eqref{v} on $[0,1]$ with $p$ replaced by $p_1$, satisfies $v_1(1)=0$.
%We set $u_0= v_0 + \varphi_1$. We denote by $u$ the solution of \eqref{u} and set $v=u-\varphi_1$. Then $v$ is the solution of \eqref{v}.  

\smallskip\noindent
{\bf Second step.} Let $v_0\in X$ be the initial condition of \eqref{v}. We decompose $v_0$ as follows
\begin{equation*}
v_0=\langle v_0,\varphi_1\rangle\varphi_1+v_{0,1},
\end{equation*}
where $v_{0,1}\in \varphi_1^\perp$ and we suppose that $\left|\langle v_0,\varphi_1\rangle\right| < r_1$. We define $t_{R}$ as
\begin{equation}\label{trR}
t_{R}:=\frac{1}{2\lambda_2}\log{\left(\frac{R^2}{r_1^2}\right)}
\end{equation} 
and in the time interval $[0,t_{R}]$ we take the control $p\equiv0$. Then, for all $t\in [0,t_{R}]$, we have that
\begin{equation*}
||v(t)||^2\leq \left|\left|e^{-tA}\left(\langle v_0,\varphi_1\rangle\varphi_1+v_{0,1}\right)\right|\right|^2\leq \left|\langle v_0,\varphi_1\rangle\right|^2+e^{-2\lambda_2t}\left|\left|v_{0,1}\right|\right|^2 < r_1^2+e^{-2\lambda_2t}R^2.
\end{equation*}
In particular, for $t=t_{R}$, it holds that $||v(t_{R})||^2 < 2 r^2_1$.

Now, we define $T_{R}:=t_{R}+1$ and set $v_1(0)=v(t_R)$. Thanks to the first step of the proof, there exists a control $p_1\in L^2(0,1)$, such that $v_1(1)=0$, where $v_1$ is the solution of \eqref{v} on $[0,1]$ with $p$ replaced by $p_1$.

Then $v(t)=v_1(t-t_R)$ solves \eqref{v} in the time interval $(t_{R},T_{R}]$ with the control $p_1(t-t_{R})$ that steers the solution $v$ to $0$ at $T_{R}$.
\end{proof}
\begin{proof}[Proof (of Theorem \ref{teoglobal})]
We start with the case $\lambda_1=0$. Let $u_0\in X$ that satisfies \eqref{ipu0}. Set $v(t):=u(t)-\varphi_1$, then $v$ satisfies \eqref{v} and moreover $v_0:=v(0)=u_0-\varphi_0$ fulfills \eqref{ipv0}. Thus, by Lemma \ref{lemglobal}, problem \eqref{u} is exactly controllable to the ground state solution $\psi_1
\equiv\varphi_1$ in time $T_{R}$.

Now, we consider the case $\lambda_1>0$. As in the proof of Theorem \ref{teo1}, we introduce the variable $z(t)=e^{\lambda_1t}u(t)$ that solves problem \eqref{z}. For such a system, since the first eigenvalue of $A_1$ is equal $0$, we have the exact controllability to $\varphi_1$ in time $T_{R}$. Namely $z(T_{R})=\varphi_1$, that is equivalent to the exact controllability of $u$ to $\psi_1$:
%\begin{equation}
%\begin{split}
%z(T_{R})&=\varphi_1\\
%&\,\,\,\rotatebox[origin=c]{-90}{$\sse$}\\
%e^{\lambda_1T_{R}}u(T_{R})&=\varphi_1\\
%&\,\,\,\rotatebox[origin=c]{-90}{$\sse$}\\
%u(T_{R})&=\psi_1(T_{R}).
%\end{split}
%\end{equation}

\begin{equation}
z(T_{R})=\varphi_1\
\quad\Longleftrightarrow\quad
e^{\lambda_1T_{R}}u(T_{R})=\varphi_1
\quad\Longleftrightarrow\quad
u(T_{R})=\psi_1(T_{R}).
\end{equation}
The proof is thus complete.
\end{proof}

The proof of Theorem \ref{teoglobal0} easily follows from Theorem \eqref{teoglobal}.
\begin{proof}[Proof (of Theorem \ref{teoglobal0})]
Suppose that $\gamma:=\langle u_0,\varphi_1\rangle\neq0$. We decompose $u_0$ as $u_0=\gamma\varphi_1+\zeta_1$, with $\zeta_1:=u_0-\langle u_0,\varphi_1\rangle\varphi_1\in\varphi_1^\perp$ and define $\tilde{u}(t):=u(t)/\gamma$. Hence, $\tilde{u}$ solves
\begin{equation}\label{utildeglobal}
\left\{
\begin{array}{ll}
\tilde{u}'(t)+A\tilde{u}(t)+p(t)B\tilde{u}(t)=0,& t>0\\
\tilde{u}(0)=\varphi_1+\tilde{\zeta_1},
\end{array}\right.
\end{equation}
where $\tilde{\zeta_1}:=\zeta_1/\gamma$.

We apply Theorem \ref{teoglobal} to \eqref{utildeglobal} to deduce the existence of $T_R>0$ such that $\tilde{u}(T_R)=\psi_1(T_R)$. Therefore, the solution of \eqref{u} with initial condition $u_0\in X$ that do not vanish along the direction $\varphi_1$ can be exactly controlled in time $T_R$ to the trajectory $\langle u_0,\varphi_1\rangle\psi_1(\cdot)$.

Note that if $u_0\in X$ satisfies both $u_0\in\varphi_1^\perp$ and \eqref{cone}, then we have trivially that $u_0\equiv 0$. We then choose $p\equiv 0$, so that the solution of \eqref{u} remains constantly equal to $\phi_1\equiv 0$.
\end{proof}

\section{Applications}
In this section we present some examples of parabolic equations for which Theorem \ref{teo1} can be applied. The hypotheses \eqref{ipA}--\eqref{ipB} have been verified in \cite{acu} and \cite{cu}, to which we refer for more details. Furthermore, we observe that also the global results Theorem \ref{teoglobal} and Theorem \ref{teoglobal0} can be applied to any example.
\subsection{Diffusion equation with Dirichlet boundary conditions}\label{ex1}
Let $I=(0,1)$ and $X=L^2(0,1)$. Consider the following problem
\begin{equation}\label{eqex1}\left\{\begin{array}{ll}
u_t(t,x)-u_{xx}(t,x)+p(t)\mu(x)u(t,x)=0 & x\in I,t>0 \\\\
u(t,0)=0,\,\,u(t,1)=0, & t>0\\\\
u(0,x)=u_0(x) & x\in I.
\end{array}\right.
\end{equation}
We denote by $A$ the operator defined by
\begin{equation*}
D(A)=H^2\cap H^1_0(I),\quad A\varphi=-\frac{d^2\varphi}{dx^2}.
\end{equation*}
and it can be checked that $A$ satisfies \eqref{ipA}. We indicate by $\{\lambda_k\}_{k\in\NN^*}$ and $\{\varphi_k\}_{k\in\NN^*}$ the families of eigenvalues and eigenfunctions of $A$, respectively:
\begin{equation*}
\lambda_k=(k\pi)^2,\quad \varphi_k(x)=\sqrt{2}\sin(k\pi x),\quad \forall k\in\NN^*.
\end{equation*}

It is easy to see that \eqref{gap} holds true:
\begin{equation*}
\sqrt{\lambda_{k+1}}-\sqrt{\lambda_k}=\pi,\qquad \forall k\in \NN^*.
\end{equation*}

Let $B:X\to X$ be the operator
\begin{equation*}
B\varphi=\mu\varphi
\end{equation*}
with $\mu\in H^3(I)$ such that
\begin{equation}\label{mu}
\mu'(1)\pm\mu'(0)\neq 0\quad\mbox{ and }\quad\langle\mu\varphi_1,\varphi_k\rangle\neq0\quad\forall k \in \NN^*.
\end{equation}
Then, there exists $b>0$ such that
\begin{equation*}
\lambda_k^{3/2}|\langle \mu\varphi_1,\varphi_k\rangle|\geq b,\qquad\forall k\in\NN^*.
\end{equation*}
For instance, a suitable function that satisfies \eqref{mu} is $\mu(x)=x^2$, for which $b=\frac{2\pi^2-3}{6\pi^2}$.

For any $T>0$, we define $T_f$ as in \eqref{Talpha}. Then, there exists a constant $R_{T_f}>0$ such that the solution $u$ of \eqref{eqex1}, with $u_0\in B_{R_{T_f}}(\varphi_1)$, reaches the ground state solution $\psi_1(t,x)=\sqrt{2}\sin(\pi x)e^{-\pi^2t}$ in time $T_f$ and stays on it forever.
\subsection{Diffusion equation with Neumann boundary conditions}\label{ex2}
Let $I=(0,1)$, $X=L^2(I)$ and consider the Cauchy problem
\begin{equation}\label{eqex2}
\left\{\begin{array}{ll}
u_t(t,x)-u_{xx}(t,x)+p(t)\mu(x)u(t,x)=0 & x\in I,t>0 \\\\
u_x(t,0)=0,\,\,u_x(t,1)=0, &t>0\\\\
u(0,x)=u_0(x). & x\in I.
\end{array}\right.
\end{equation}
The operator $A$, defined by
\begin{equation*}
D(A)=\{ \varphi\in H^2(0,1): \varphi'(0)=0,\,\,\varphi'(1)=0\},\quad A\varphi=-\frac{d^2\varphi}{dx^2}
\end{equation*}
satisfies \eqref{ipA} and has the following eigenvalues and eigenfunctions 
\begin{equation*}
\begin{array}{lll}
\lambda_0=0,&\varphi_0=1\\
\lambda_k=(k\pi)^2,& \varphi_k(x)=\sqrt{2}\cos(k\pi x),& \forall k\geq1.
\end{array}
\end{equation*}
Thus, the gap condition \eqref{gap} is fulfilled with $\alpha=\pi$. The ground state solution is just the stationary function $\psi_1(x)=\varphi_1(x)=1$.

We define $B:X\to X$ as the multiplication operator by a function $\mu\in H^2(I)$, $B\varphi=\mu\varphi$, such that
\begin{equation}
\mu'(1)\pm\mu'(0)\neq 0\quad\mbox{ and }\quad\langle\mu,\varphi_k\rangle\neq0\quad\forall k \in \NN.
\end{equation}
It can be proved that, there exists $b>0$ such that
\begin{equation}\label{mu2}
\lambda_k|\langle \mu\varphi_0,\varphi_k\rangle|\geq b,\qquad\forall k\in\NN^*.
\end{equation}
For example, $\mu(x)=x^2$ satisfies \eqref{mu2} with $b=2\sqrt{2}$.

Therefore, equation \eqref{eqex2} is controllable to the ground state solution $\psi_1=1$ in any time $T>0$ as long as $u_0\in B_{R_T}(1)$, with $R_T>0$ a suitable constant.
\subsection{Variable coefficient parabolic equation}\label{ex3}
Let $I=(0,1)$, $X=L^2(I)$ and consider the problem
\begin{equation}\label{eqex3}
\left\{
\begin{array}{ll}
u_t(t,x)-((1+x)^2u_x(t,x))_x+p(t)\mu(x)u(t,x)=0&x\in I,t>0\\\\
u(t,0)=0,\quad u(t,1)=0,&t>0\\\\
u(0,x)=u_0(x)&x\in I.
\end{array}
\right.
\end{equation}
We denote by $A:D(A)\subset X\to X$ the following operator
\begin{equation*}
D(A)=H^2\cap H^1_0(I),\qquad A\varphi=-((1+x)^2\varphi_x)_x.
\end{equation*}
It can be checked that $A$ satisfies \eqref{ipA} and that the eigenvalues and eigenfunctions have the following expression
\begin{equation*}
\lambda_k=\frac{1}{4}+\left(\frac{k\pi}{\ln2}\right)^2,\qquad\varphi_k=\sqrt{\frac{2}{\ln 2}}(1+x)^{-1/2}\sin\left(\frac{k\pi}{\ln2 }\ln(1+x)\right).
\end{equation*}
Furthermore, $\{\lambda_k\}_{k\in\NN^*}$ verifies the gap condition \eqref{gap} with $\alpha=\pi/\ln{2}$.

We define the operator $B:X\to X$ by $B\varphi=\mu\varphi$, where $\mu\in H^2(I)$ is such that
\begin{equation}\label{mu3}
2\mu'(1)\pm\mu'(0)\neq0,\quad\mbox{and}\quad \langle \mu\varphi_1,\varphi_k\rangle\neq0\quad\forall k \in\NN^*.
\end{equation}
Hence, thanks to \eqref{mu3}, \eqref{ipB} is fulfilled with $q=3/2$. An example of a suitable function $\mu$ that satisfies \eqref{mu3} is $\mu(x)=x$, see \cite{acu} for the verification.

Thus, from Theorem \ref{teo1}, we deduce that, for any $T>0$, system \eqref{eqex3} is controllable to the ground state solution if the initial condition $u_0$ is close enough to $\varphi_1$. 
\subsection{Diffusion equation in a $3D$ ball with radial data}\label{ex4}
In this example, we study the controllability of an evolution equation in the three dimensional unit ball $B^3$ for radial data. The bilinear control problem is the following
\begin{equation}\label{eqex4}
\left\{\begin{array}{ll}
u_t(t,r)-\Delta u(t,r)+p(t)\mu(r)u(t,r)=0 & r\in[0,1], t>0 \\\\
u(t,1)=0,&t>0\\\\
u(0,r)=u_0(r) & r\in[0,1]
\end{array}\right.
\end{equation}
where the Laplacian in polar coordinates for radial data is given by the following expression
$$\Delta\varphi(r)=\partial^2_r \varphi(r)+\frac{2}{r}\partial_r\varphi(r).$$
The function $\mu$ is a radial function as well in the space $H^3_r(B^3)$, where the spaces $H^k_r(B^3)$ are defined as follows
$$X:=L^2_{r}(B^3)=\left\{\varphi\in L^2(B^3)\,|\, \exists \psi:\RR\to\RR, \varphi(x)=\psi(|x|)\right\}$$
$$H^k_r(B^3):=H^k(B^3)\cap L^2_{r}(B^3) .$$

The domain of the Dirichlet Laplacian $A:=-\Delta$ in $X$ is $D(A)=H^2_{r}\cap H^1_0(B^3)$. We observe that $A$ satisfies hypothesis \eqref{ipA}. We denote by $\{\lambda_k\}_{k\in\NN^*}$ and $\{\varphi_k\}_{k\in\NN^*}$ the families of eigenvalues and eigenvectors of $A$, $A\varphi_k=\lambda_k\varphi_k$, namely
\begin{equation}\label{ee}\varphi_k=\frac{\sin(k\pi r)}{\sqrt{2\pi}r},\qquad\lambda_k=(k\pi)^2
\end{equation}
$\forall k\in\NN^*$, see \cite[Section 8.14]{leb}. Since the eigenvalues of $A$ are actually the same of the Dirichlet $1D$ Laplacian, \eqref{gap} is satisfied, as we have seen in Example \ref{ex1}.

Let $B:X\to X$ be the multiplication operator $Bu(t,r)=\mu(r)u(t,r)$, with $\mu$ be such that
\begin{equation}\label{mu4}
\mu'(1)\pm\mu'(0)\neq 0,\quad\mbox{and}\quad \langle \mu\varphi_1,\varphi_k\rangle\neq0\quad\forall k\in\NN^*.
\end{equation}

Then, it can be proved that 
\begin{equation}\label{mu4p}
\lambda_k^{3/2}|\langle \mu\varphi_1,\varphi_k\rangle|\geq b,\qquad \forall k\in\NN^*,
\end{equation}
with $b$ a positive constant. For instance, $\mu(x)=x^2$ verifies \eqref{mu4} and \eqref{mu4p} with $b=\frac{2\pi^2-3}{6\pi^2}$.

Therefore, by applying Theorem \ref{teo1}, we conclude that for any $T>0$, the exists a suitable constant $R_T>0$ such that, if $u_0\in B_{R_T}(\varphi_1)$, problem \eqref{eqex4} is exactly controllable to the ground state $\psi_1$ in time $T$.

\subsection{Degenerate parabolic equation}\label{ex5}
In this last section we want to address an example of a control problem for a degenerate evolution equation of the form
\begin{equation}\label{eqex5}
\left\{
\begin{array}{ll}
u_t-\left(x^{\gamma} u_x\right)_x+p(t)x^{2-\gamma}u=0,& (t,x)\in (0,+\infty)\times(0,1)\\\\
u(t,1)=0,\quad\left\{\begin{array}{ll} u(t,0)=0,& \mbox{ if }\gamma\in[0,1),\\\\ \left(x^{\gamma}u_x\right)(t,0)=0,& \mbox{ if }\gamma\in[1,3/2),\end{array}\right.\\\\
u(0,x)=u_0(x).
\end{array}
\right.
\end{equation}
where $\gamma\in[0,3/2)$ describes the degeneracy magnitude, for which Theorem \ref{teo1} applies. 

If $\gamma\in[0,1)$ problem \eqref{eqex5} is called weakly degenerate and the natural spaces for the well-posedness are the following weighted Sobolev spaces. Let $I=(0,1)$ and $X=L^2(I)$, we define
\begin{equation*}
\begin{array}{l}
H^1_{\gamma}(I)=\left\{u\in X: u \mbox{ is absolutely continuous on } [0,1], x^{\gamma/2}u_x\in X\right\}\\\\
H^1_{\gamma,0}(I)=\left\{u\in H^1_\gamma(I):\,\, u(0)=0,\,\,u(1)=0\right\}\\\\
H^2_\gamma(I)=\left\{u\in H^1_\gamma(I): x^{\gamma}u_x\in H^1(I)\right\}.
\end{array}
\end{equation*}
We denote by $A:D(A)\subset X\to X$ the linear degenerate second order operator
\begin{equation}
\left\{\begin{array}{l}
\forall u\in D(A),\quad Au:=-(x^{\gamma}u_x)_x,\\\\
D(A):=\{u\in H^1_{\gamma,0}(I),\,\, x^{\gamma}u_x\in H^1(I)\}.
\end{array}\right.
\end{equation}
It is possible to prove that $A$ satisfies \eqref{ipA} (see, for instance \cite{cmp}) and furthermore, if we denote by $\{\lambda_k\}_{k\in\NN^*}$  the eigenvalues and by $\{\varphi_k\}_{k\in\NN^*}$ the corresponding eigenfunctions, it turns out that the gap condition \eqref{gap} is fulfilled with $\alpha=\frac{7}{16}\pi$ (see \cite{kl}, page 135). 

If $\gamma\in[1,3/2)$, problem \eqref{eqex5} is called strong degenerate and the corresponding weighted Sobolev space are described as follows: given $I=(0,1)$ and $X=L^2(I)$, we define
\begin{equation*}
\begin{array}{l}
H^1_{\gamma}(I)=\left\{u\in X: u \mbox{ is absolutely continuous on } (0,1],\,\, x^{\gamma/2}u_x\in X\right\}\vspace{.1cm}\\\\
H^1_{\gamma,0}(I):=\left\{u\in H^1_{\gamma}(I):\,\,u(1)=0\right\},\vspace{.1cm}\\\\
H^2_\gamma(I)=\left\{u\in H^1_\gamma(I):\,\, x^{\gamma}u_x\in H^1(I)\right\}.
\end{array}
\end{equation*}
In this case the operator $A:D(A)\subset X\to X$ is defined by
\begin{equation*}
\left\{\begin{array}{l}
\forall u\in D(A),\quad Au:=-(x^{\gamma}u_x)_x,\vspace{.1cm}\\\\
D(A):=\left\{u\in H^1_{\gamma,0}(I):\,\, x^{\gamma}u_x\in H^1(I)\right\}\vspace{.1cm}\\
\qquad\,\,\,\,\,=\left\{u\in X:\,\,u \mbox{ is absolutely continuous in (0,1] },\,\, x^{\gamma}u\in H^1_0(I),\right.\vspace{.1cm}\\
\qquad\qquad\,\,\,\left.x^{\gamma}u_x\in H^1(I)\mbox{ and } (x^{\gamma}u_x)(0)=0\right\}
\end{array}\right.
\end{equation*}
and it has been proved that \eqref{ipA} holds true (see, for instance \cite{cmvn}) and that \eqref{gap} is satisfied for $\alpha=\frac{\pi}{2}$ (see \cite{kl}).

For all $\gamma\in[0,3/2)$, we define the linear operator $B:X\to X$ by $Bu(t,x)=x^{2-\gamma}u(t,x)$ and in \cite{cu} we have proved that there exists a constant $b>0$ such that
\begin{equation*}
\lambda_k^{3/2}|\langle B\varphi_1,\varphi_k\rangle|\geq b\quad\forall k\in\NN^*.
\end{equation*}

Finally, by applying Theorem \ref{teo1}, we ensure the exact controllability of problem \eqref{eqex5} to the ground state solution, for both weakly and strongly degenerate problems.
\section*{Acknowledgments}
We are grateful to J. M. Coron and P. Martinez for their precious comments and suggestions. 
\section*{References}
\bibliographystyle{plain}
\bibliography{biblio}

\begin{thebibliography}{10}

\bibitem{acu}
F.~Alabau-Boussouira, P.~Cannarsa, and C.~Urbani.
\newblock Superexponential stabilizability of evolution equations of parabolic
  type via bilinear control.
\newblock {\em preprint available on arXiv:1910.06802}.

\bibitem{bms}
J.M. Ball, J.E. Marsden, and M.~Slemrod.
\newblock Controllability for distributed bilinear systems.
\newblock {\em SIAM Journal on Control and Optimization}, 20(4):575--597, 1982.

\bibitem{beau}
K.~Beauchard.
\newblock Local controllability of a 1-{D} {S}chr\"{o}dinger equation.
\newblock {\em J. Math. Pures Appl. (9)}, 84(7):851--956, 2005.

\bibitem{b}
K.~Beauchard.
\newblock Local controllability and non-controllability for a 1d wave equation
  with bilinear control.
\newblock {\em Journal of Differential Equations}, 250(4):2064--2098, 2011.

\bibitem{beaucor}
K.~Beauchard and J.-M. Coron.
\newblock Controllability of a quantum particle in a moving potential well.
\newblock {\em J. Funct. Anal.}, 232(2):328--389, 2006.

\bibitem{bl}
K.~Beauchard and C.~Laurent.
\newblock Local controllability of 1d linear and nonlinear {S}chr{\"o}dinger
  equations with bilinear control.
\newblock {\em J. Math. Pures Appl.}, 94:520--554, 2010.

\bibitem{bm}
K.~Beauchard and M.~Morancey.
\newblock Local controllability of 1d {S}chr{\"o}dinger equations with bilinear
  control and minimal time.
\newblock {\em Math. Control Relat. Fields}, 4(2):125--160, 2014.

\bibitem{cmp}
M.~Campiti, G.~Metafune, and D.~Pallara.
\newblock Degenerate self-adjoint evolution equations on the unit interval.
\newblock In {\em Semigroup Forum}, volume~57, pages 1--36. Springer, 1998.

\bibitem{cfk}
P.~Cannarsa, G.~Floridia, and A.~Y. Khapalov.
\newblock Multiplicative controllability for semilinear reaction-diffusion
  equations with finitely many changes of sign.
\newblock {\em Journal de Math{\'e}matiques Pures et Appliqu{\'e}es},
  108(4):425--458, 2017.

\bibitem{ck}
P.~Cannarsa and A.Y. Khapalov.
\newblock Multiplicative controllability for reaction-diffusion equations with
  target states admitting finitely many changes of sign.
\newblock {\em Discrete Contin. Dyn. Syst. Ser. B}, 14:1293--1311, 2010.

\bibitem{cmvn}
P.~Cannarsa, P.~Martinez, and J.~Vancostenoble.
\newblock Carleman estimate for a class of degenerate parabolic operators.
\newblock {\em SIAM Journal on Control and Optimization}, 47(1):1--19, 2008.

\bibitem{cmv}
P.~Cannarsa, P.~Martinez, and J.~Vancostenoble.
\newblock The cost of controlling weakly degenerate parabolic equations by
  boundary controls.
\newblock {\em Mathematical Control \& Related Fields}, 7(2):171--211, 2017.

\bibitem{cu}
P.~Cannarsa and C.~Urbani.
\newblock Superexponential stabilizability of degenerate parabolic equations
  via bilinear control.
\newblock {\em preprint available on arXiv:1910.06198}.

\bibitem{cor}
J.-M. Coron.
\newblock On the small-time local controllability of a quantum particle in a
  moving one-dimensional infinite square potential well.
\newblock {\em C. R. Math. Acad. Sci. Paris}, 342(2):103--108, 2006.

\bibitem{fgip}
E.~Fern{\'a}ndez-Cara, S.~Guerrero, O~Y. Imanuvilov, and J.-P. Puel.
\newblock Local exact controllability of the navier--stokes system.
\newblock {\em Journal de math{\'e}matiques pures et appliqu{\'e}es},
  83(12):1501--1542, 2004.

\bibitem{kl}
V.~Komornik and P.~Loreti.
\newblock {\em Fourier series in control theory}.
\newblock Springer Science \& Business Media, 2005.

\bibitem{leb}
N.N. Lebedev.
\newblock {\em Special functions and their applications. Revised English
  edition. Translated and edited by Richard A. Silverman}.
\newblock Prentice-Hall Inc., Englewood Cliffs, NJ, 1965.

\end{thebibliography}
\end{document}